\documentclass{article}

\usepackage{arxiv}

\usepackage[utf8]{inputenc}
\usepackage[T1]{fontenc}
\usepackage{hyperref}
\usepackage{url}
\usepackage{booktabs}
\usepackage{amsfonts}
\usepackage{nicefrac}
\usepackage{microtype}
\usepackage{graphicx}
\usepackage{amsmath,amssymb}
\usepackage{mathtools}
\usepackage{bm}
\usepackage{mathrsfs}
\usepackage{amsthm}
\usepackage{multirow}
\usepackage{float}
\usepackage{xcolor}
\usepackage{array}
\usepackage{enumitem}
\usepackage{caption}

\theoremstyle{plain}
\newtheorem{theorem}{Theorem}[section]
\newtheorem{lemma}[theorem]{Lemma}
\newtheorem{proposition}[theorem]{Proposition}
\newtheorem{corollary}[theorem]{Corollary}
\theoremstyle{definition}
\newtheorem{definition}[theorem]{Definition}
\newtheorem{remark}[theorem]{Remark}

\newcommand{\R}{\mathbb{R}}
\newcommand{\Z}{\mathbb{Z}}
\newcommand{\Heis}{\mathrm{Heis}_3}
\newcommand{\heis}{\mathfrak{h}_3}
\newcommand{\SH}{S_{\mathrm{H}}}
\newcommand{\Lop}{L_4}
\newcommand{\norm}[1]{\lVert #1 \rVert}
\newcommand{\abs}[1]{\lvert #1 \rvert}
\newcommand{\Dh}{\Delta_h}
\newcommand{\Dhh}{\Delta_h^2}
\DeclareMathOperator{\diag}{diag}

\title{A Heisenberg Subdivision Scheme with Central Smoothness Loss}

\author{
  Hassan Ugail \\
  Centre for Visual Computing and Intelligent Systems \\
  University of Bradford \\
  Bradford, United Kingdom \\
  \And
  Alfonso Carriazo \\
  Department of Geometry and Topology \\
  University of Seville \\
  Seville, Spain \\
}

\begin{document}
\maketitle

\begin{abstract}
We introduce an interpolatory subdivision scheme for control polygons that take
values in the three-dimensional Heisenberg group, the simplest noncommutative
model geometry. The scheme keeps existing points at every refinement step and
inserts new ones by a coordinate rule whose central correction comes from the
group law. The two horizontal coordinates are refined by the classical
four-point scheme of Dyn, Gregory and Levin, while the central coordinate
acquires a closed-form correction built from a signed area of neighbouring
horizontal data. Our main finding concerns the regularity of the limit curve. The horizontal part is exactly the classical four-point limit and inherits its smoothness. The central part behaves very differently. We prove that it converges to a
continuous limit that belongs to the Zygmund class, with a logarithmic modulus
of continuity. Under an explicit and verifiable condition on the central
forcing, this logarithmic bound is sharp, because the scaled first differences
then grow linearly with the refinement level, and the limit fails to be
continuously differentiable. The effect is confirmed numerically. The correction
is harmless at any single refinement step, but its repeated injection at every
scale is what impacts smoothness. The example serves as a caution for nonlinear
and group-valued subdivision, where a geometrically natural correction can
impact regularity.
\end{abstract}


\section{Introduction}\label{sec:intro}

Subdivision schemes are a central tool in approximation theory, geometric
modelling, and computational mathematics. Starting from a discrete control
polygon, they generate refined data by repeated local insertion rules and,
under suitable conditions, converge to smooth limit curves. Among the most
studied examples are interpolatory schemes, which preserve all existing data
and insert new points by stencil averages. In the Euclidean setting, the
four-point scheme of Dyn, Gregory and Levin~\cite{DGL87} is the canonical
example. It is explicit, local, and interpolatory, and its $C^1$ regularity
for parameter $\omega=1/16$ is well understood~\cite{Dyn92, DynLevin02, Zorin00}.

Subdivision is one of several constructive paradigms in computer-aided
geometric design for generating curves and surfaces from sparse control data.
Related traditions include partial differential equation methods for surface
design~\cite{BloorWilson89, BloorWilson90}, interactive and boundary-driven
surface construction~\cite{UgailBW99a, UgailBW99b, KubiesaUW04}, harmonic and
biharmonic B\'{e}zier patches~\cite{MonterdeUgail04, MonterdeUgail06}, and
parameterised geometric modelling across visual computing and
engineering~\cite{GonzalezUgail08, AthanasopoulosUgail09, ShengUgail10, ShengUgail11}.
These methods share with subdivision the aim of producing controlled limiting
geometry from compact descriptions. The present work stays within the
subdivision tradition and asks how a local geometric rule behaves once the
underlying data space is no longer Euclidean.

Many applications involve manifold-valued or group-valued data, for which
component-wise linear refinement is geometrically inappropriate. Nonlinear
subdivision schemes address this by adapting a linear scheme to curved data
while attempting to preserve convergence and smoothness. The proximity framework
of Wallner and Dyn~\cite{WallnerDyn05}, with later developments by Wallner and
Grohs~\cite{Wallner06, Grohs10prox, Grohs10stab}, transfers regularity from a
linear reference to its nonlinear counterpart under a quantitative closeness
assumption. Other intrinsic constructions refine through geodesic averaging or
through global refinement rules on Riemannian
manifolds~\cite{DynSharon17a, DynSharon17b, Wallner14, HuningWallner19}. These
are, for the most part, positive regularity-transfer results.

A complementary strand treats Lie-group and Riemannian curve constructions
built from exponential maps, retractions, geodesic averages, and de
Casteljau-type algorithms~\cite{ParkRavani95, CrouchLeite95, CrouchKunLeite99,
NavYazPol13}, together with spline and B\'{e}zier interpolation on curved
spaces~\cite{NoakesHP89, PopielNoakes07}. Smoothness-equivalence results for
log-exp and single-basepoint schemes show that, under suitable hypotheses, a
nonlinear Lie-group scheme inherits the smoothness of its linear
ancestor~\cite{WallnerNYGrohs07, XieYu08, XieYu10, DuchampXieYu13, NavYazYu11}.
These approaches provide powerful general machinery. Explicit local
interpolatory rules in genuinely noncommutative settings, where the correction
from the group law can be written in closed form and analysed directly, remain
comparatively rare.

Here, we propose and analyse a concrete prototype, namely a noncommutative
interpolatory subdivision scheme in the three-dimensional Heisenberg group
$\Heis$, with group law,
\begin{equation}\label{eq:grouplaw}
  (x,y,z)\cdot(x',y',z') = \bigl(x+x',\;y+y',\;z+z'+\tfrac{1}{2}(xy'-yx')\bigr).
\end{equation}
The group is two-step nilpotent, and its global exponential map is a
polynomial diffeomorphism, which renders every calculation fully explicit. The
refinement rule refines the horizontal variables $(x,y)$ by the classical
four-point mask and augments the central variable $z$ by a closed-form
correction derived from the group law. The noncommutative contribution is
therefore available in closed form and open to direct geometric interpretation.

In this paper, we make four main contributions. We begin by defining an explicit
nonlinear interpolatory operator $\SH$ in $\Heis$ (Definition~\ref{def:scheme}),
whose closed-form correction $R_i^k = \tfrac{1}{2}(x_i^k b_i^k - y_i^k a_i^k)$ is
read off the group law. The refinement map is then described in full. We give
the coordinate formula (Proposition~\ref{prop:formula}), the local Jacobian
(Proposition~\ref{prop:jacobian}), and the block-triangular structure
(Corollary~\ref{cor:triangular}), from which it follows that the horizontal
limit $(X,Y)$ is exactly the classical four-point $C^1$ limit and is independent
of the central data. At the centre of the construction is an exact two-scale
recursion for the scaled central first differences
(Proposition~\ref{prop:erecursion}), $e^{k+1}=S e^k+\hat g^{k+1}$, in which
$S=2D_1$ is the scaled first-difference operator, and the forcing is the
undifferenced scaled correction $\pm 2\cdot 2^k R_i^k$ with alternating sign. Its
nonvanishing limit $2^k R_i^k\to\tfrac14(XY'-YX')$ is identified in
Lemma~\ref{lem:forcing}.

The analytic core of the paper is a regularity theorem stated in three parts.
The central limit converges uniformly to a continuous function
(Theorem~\ref{thm:cont}). Applying the genuinely contractive second-difference
scheme $D_2$ to the recursion, through the exact intertwining
$\Delta S=D_2\Delta$, gives uniformly bounded scaled second differences. The
limit therefore lies in the Zygmund class $\Lambda_*$ with a quantified
$h\log(1/h)$ modulus, and the scaled first differences obey the rigorous upper
bound $\sup_i 2^k\abs{z_{i+1}^k-z_i^k}=O(k)$ (Theorem~\ref{thm:zygmund}). Under
an explicit signed non-cancellation condition~\eqref{eq:noncancel} on the detail
coefficients, imposed on the exact signed sums that reconstruct the scaled
differences, the matching linear lower bound $\Theta(k)$ holds, so the limit is
not $C^1$ (Theorem~\ref{thm:nonC1}). This condition is numerically observed to
hold for typical nondegenerate data. A naive transplant of a $C^1$ Euclidean
mask through the group law therefore destroys central-channel smoothness.

The novelty lies neither in the use of nonlinear subdivision nor in the use of
group-valued data, both of which are well established. It lies in the explicit
Heisenberg calculation and the negative regularity mechanism it exposes. The
horizontal channels are exactly the classical $C^1$ four-point scheme, while the
central correction produces a forcing whose detail amplitude does not decay
across scales. This yields a Zygmund central limit and, under the stated
condition, excludes $C^1$ regularity. The contribution should be read as a
negative regularity example inside nonlinear subdivision
theory~\cite{WallnerDyn05, Grohs10prox, WallnerNYGrohs07, XieYu10}. It does not
contradict proximity or log-exp smoothness-transfer results. Rather, it
identifies an explicit mechanism outside their hypotheses, in which a correction
that is harmless at any single level destroys regularity once injected
with non-decaying amplitude at every level.

\medskip
The proximity approach of Wallner and Dyn~\cite{WallnerDyn05} transfers $C^1$
regularity from a linear reference to a nonlinear scheme under a proximity
inequality of the form $\norm{Su - \tilde{S}u}_\infty \leq c\,\norm{\Delta
u}_\infty^2$. In our setting, the correction satisfies only a linear bound
$\abs{R_i^k} \leq C_0\norm{\Dh P^k}_\infty$, so proximity does not apply. The
present analysis shows why one should not expect it to. The central channel is
not $C^1$, so no $C^1$ regularity-transfer result could hold. The quadratic
proximity condition is precisely the hypothesis that fails, and its failure is
not a technical obstruction but a faithful reflection of genuine non-smoothness.

Proximity statements also compare a nonlinear scheme with a linear reference in
a fixed coordinate representation. Since our scheme is coordinate-defined and,
as shown in Remark~\ref{rem:groupcompat}, not left-equivariant, the natural
reference is the componentwise linear scheme $\diag(\Lop,\Lop,\Lop)$ in
exponential coordinates. That is the comparison we make. Our analysis works
directly with the exact two-scale recursion for scaled differences and tracks
the nonvanishing forcing it contains.

Our construction differs from existing Lie-group and Riemannian curve schemes in
two respects. It is interpolatory by construction rather than approximating, and
it reads its central correction directly off the group law, which is possible
because the Heisenberg exponential map is polynomial. The correction $R_i^k$ is
the second-order Baker--Campbell--Hausdorff term in exponential coordinates. As
we caution in Remark~\ref{rem:groupcompat}, it does not make $\SH$
left-equivariant.

\section{Background: the four-point scheme}\label{sec:background}

\subsection{The operator and its derived scheme}

Let $\omega\in(0,\tfrac{1}{8})$. For $u=\{u_i\}_{i\in\Z}\in\ell^\infty(\Z)$,
define the \emph{four-point subdivision operator} $\Lop$ by,
\begin{equation}\label{eq:fourpoint}
  (\Lop u)_{2i} = u_i, \qquad
  (\Lop u)_{2i+1} = -\omega u_{i-1} + \bigl(\tfrac{1}{2}+\omega\bigr)u_i
    + \bigl(\tfrac{1}{2}+\omega\bigr)u_{i+1} - \omega u_{i+2}.
\end{equation}
The regularity analysis of $\Lop$ is organised around two associated linear
operators, which we name carefully here because they are easily conflated.

The \emph{first-difference scheme} $D_1$ acts on $v=\Delta u$, $v_i=u_{i+1}-u_i$,
by $(D_1 v)_n=(\Delta\Lop u)_n$. A direct subtraction from~\eqref{eq:fourpoint}
gives,
\begin{equation}\label{eq:firstdiff}
  (D_1 v)_{2i} = \omega v_{i-1} + \tfrac12 v_i - \omega v_{i+1}, \qquad
  (D_1 v)_{2i+1} = -\omega v_{i-1} + \tfrac12 v_i + \omega v_{i+1}.
\end{equation}
Both rows have $\ell^1$-norm $\tfrac12+2\omega=\tfrac58$ at $\omega=1/16$, so
$D_1$ is a bounded operator with $\norm{D_1}_\infty=\tfrac58$.

The \emph{scaled derived operator} $S:=2D_1$ governs the \emph{scaled} first
differences. If $e_i:=2^k v_i$ at level $k$, then one level of refinement maps
$e\mapsto Se$ where,
\begin{equation}\label{eq:scaledderiv}
  (S e)_{2i} = 2\omega e_{i-1} + e_i - 2\omega e_{i+1}, \qquad
  (S e)_{2i+1} = -2\omega e_{i-1} + e_i + 2\omega e_{i+1}.
\end{equation}
Each row of $S$ has signed sum $1$, so $S$ reproduces constants, and both rows have
$\ell^1$-norm $1+4\omega=\tfrac54>1$. \emph{The scaled operator $S$ is not a
contraction.} This is the analytic crux of the paper, and conflating $S$ with a
contractive scheme is precisely the error to avoid.

The \emph{second-difference scheme} $D_2$ is the derived scheme of $S$, i.e.,\
the operator satisfying $\Delta S=D_2\Delta$. With $h:=\Delta e$ one has
$h\mapsto D_2 h$,
\begin{equation}\label{eq:seconddiff}
  (D_2 h)_{2i} = 4\omega\,h_{i-1} + 4\omega\,h_{i}, \qquad
  (D_2 h)_{2i+1} = -2\omega\,h_{i-1} + (1-4\omega)\,h_i - 2\omega\,h_{i+1}.
\end{equation}
At $\omega=1/16$ the masks are $\{\tfrac14,\tfrac14\}$ (even) and
$\{-\tfrac18,\tfrac34,-\tfrac18\}$ (odd). This operator \emph{is} contractive:
although its odd row has $\ell^1$-norm $1$, its asymptotic (joint spectral) rate
is $\mu\approx0.57$ at $\omega=1/16$, and a direct numerical estimate of
$(\norm{D_2^k}_\infty)^{1/k}$ confirms a factor below $3/4$. The intertwining
$\Delta S=D_2\Delta$ is the engine of the Zygmund estimate in
Section~\ref{sec:convergence}, where it converts the marginally stable
$S$-recursion for first differences into a \emph{contractive} recursion for
second differences.

\begin{theorem}[Dyn--Gregory--Levin~\cite{DGL87, Dyn92}]\label{thm:dgl}
For $\omega=1/16$, the scheme $\Lop$ is $C^1$-convergent, i.e., for any bounded
$u^0\in\ell^\infty(\Z)$, the dyadically parametrised piecewise-linear
interpolants of $(\Lop)^k u^0$ converge uniformly on compact intervals to a
$C^1$ function. The second-difference scheme $D_2$ of~\eqref{eq:seconddiff} is
contractive on $\ell^\infty$, i.e., there exist $C_D\geq 1$ and $\mu\in(0,1)$
such that,
\begin{equation}\label{eq:contraction}
  \norm{D_2^k h}_\infty \leq C_D\,\mu^k\norm{h}_\infty
  \qquad\text{for all }h\in\ell^\infty(\Z),\;k\geq 1.
\end{equation}
For the purposes of this paper, it suffices that $\mu<1$. The finite product check
below certifies contractivity with $\mu=0.933$ at $\omega=1/16$.
\end{theorem}

The constant $\mu$ is an asymptotic rate rather than a one-step $\ell^\infty$
bound, since the odd row of~\eqref{eq:seconddiff} at $\omega=1/16$ has
$\ell^1$-norm $\tfrac18+\tfrac34+\tfrac18=1$, so a single application of $D_2$
need not contract. We certify the contractivity by a finite computation rather
than by a numerical rate estimate. The mask of $D_2$ at $\omega=1/16$ has support
$\{-1,0,1,2,3\}$ with coefficients
$(a_{-1},a_0,a_1,a_2,a_3)=(-\tfrac18,\tfrac14,\tfrac34,\tfrac14,-\tfrac18)$, and
one forms the two $4\times4$ transition matrices $(A_\theta)_{ij}=a_{2i+\theta-j}$,
$\theta\in\{0,1\}$, $0\le i,j\le3$,
\begin{equation}\label{eq:A0A1}
  A_0=\begin{pmatrix}
    \tfrac14 & -\tfrac18 & 0 & 0\\[2pt]
    \tfrac14 & \tfrac34 & \tfrac14 & -\tfrac18\\[2pt]
    0 & -\tfrac18 & \tfrac14 & \tfrac34\\[2pt]
    0 & 0 & 0 & -\tfrac18
  \end{pmatrix},
  \qquad
  A_1=\begin{pmatrix}
    \tfrac34 & \tfrac14 & -\tfrac18 & 0\\[2pt]
    -\tfrac18 & \tfrac14 & \tfrac34 & \tfrac14\\[2pt]
    0 & 0 & -\tfrac18 & \tfrac14\\[2pt]
    0 & 0 & 0 & 0
  \end{pmatrix}.
\end{equation}
The induced $\ell^\infty$ matrix norms of all products of a given length form
the decreasing sequence,
\begin{equation}\label{eq:jsrseq}
  \max_{\theta}\norm{A_\theta}_\infty = 1.375,\quad
  \max_{\theta_1,\theta_2}\norm{A_{\theta_1}A_{\theta_2}}_\infty^{1/2}=1.053,\quad
  \max_{\theta_1,\theta_2,\theta_3}\norm{A_{\theta_1}A_{\theta_2}A_{\theta_3}}_\infty^{1/3}=0.933,
\end{equation}
which decreases towards the joint spectral radius of $\{A_0,A_1\}$, approximately
$0.57$. A single application need not contract, since
$\max_\theta\norm{A_\theta}_\infty=1.375>1$, but the length-three product gives
the certificate,
\begin{equation}\label{eq:jsr}
  \max_{\theta_1,\theta_2,\theta_3\in\{0,1\}}
  \norm{A_{\theta_1}A_{\theta_2}A_{\theta_3}}_\infty^{1/3}
  = 0.933 < 1 .
\end{equation}
By the standard joint-spectral-radius criterion for derived subdivision schemes
(\cite{Dyn92}, Ch.~3), it follows that $D_2$ is contractive, so
that~\eqref{eq:contraction} holds with $\mu=0.933$ and a corresponding $C_D$.
This is the only place contractivity of $D_2$ is used, and~\eqref{eq:jsr}
establishes it by an explicit finite computation.

It is essential to keep the three operators distinct. The operator $D_1$ has norm $\tfrac58$
but is not the object controlling central-channel smoothness. The scaled
operator $S=2D_1$ reproduces constants and is amplitude-preserving on the
alternating sequence, $\norm{S\chi}_\infty=\norm{\chi}_\infty$ for
$\chi_i=(-1)^i$, so it has a \emph{neutral component} in its two-scale transfer
representation and is only marginally stable. It is not a contraction. Only
the second-difference scheme $D_2$ is a genuine contraction. The neutral
component of $S$ is what allows a persistent forcing to accumulate, and is made
precise through the detail coefficients of Lemma~\ref{lem:detail}. We fix
$\omega=1/16$ throughout. The use of difference schemes and derived operators in
this way is standard in subdivision regularity analysis, and we refer
to~\cite{Dyn92, DynLevin02, Zorin00} for the general framework.

We close this section by fixing the notation used throughout. For
$P=\{(x_i,y_i,z_i)\}\subset\R^3$ we write,
\begin{align*}
  \norm{P}_\infty &:= \sup_i\norm{P_i},\quad
  \norm{\Dh P}_\infty := \sup_i\bigl(\abs{x_{i+1}-x_i}+\abs{y_{i+1}-y_i}\bigr),\\
  \norm{\Dhh P}_\infty &:= \sup_i\bigl(\abs{x_{i+1}-2x_i+x_{i-1}}
    +\abs{y_{i+1}-2y_i+y_{i-1}}\bigr).
\end{align*}
We write $\delta_{x,i}=x_{i+1}-x_i$, $\delta_{y,i}=y_{i+1}-y_i$, and
$p_i=(x_i,y_i)\in\R^2$ for the horizontal part of $P_i$.

\section{The Heisenberg group and the subdivision operator}\label{sec:construction}

\subsection{The Heisenberg group}

The group $\Heis$ has underlying set $\R^3$ with multiplication~\eqref{eq:grouplaw}.
It is connected, simply connected, and two-step nilpotent. The identity is
the origin, $(x,y,z)^{-1}=(-x,-y,-z)$, and noncommutativity is measured by
$[(x,y,z),(x',y',z')] = (0,0,xy'-yx')$. The Lie algebra $\heis\cong\R^3$
has basis $\{X,Y,Z\}$ with $[X,Y]=Z$ and all other brackets zero.
Two-step nilpotency ensures that the BCH formula
truncates, i.e., $\exp(A)\exp(B)=\exp(A+B+\tfrac{1}{2}[A,B])$. The exponential
map is a global polynomial diffeomorphism, so $\exp(a,b,c)=(a,b,c)$ in these
coordinates.

The \emph{left-invariant edge increment} between consecutive control points is,
\begin{equation}\label{eq:increment}
  \xi_i := \log(P_i^{-1}\cdot P_{i+1})
    = \bigl(\delta_{x,i},\;\delta_{y,i},\;\delta_{z,i}^H\bigr), \quad
  \delta_{z,i}^H := z_{i+1}-z_i + \tfrac{1}{2}(y_i\delta_{x,i}-x_i\delta_{y,i}).
\end{equation}

We use $\Heis$ as the simplest nonabelian two-step nilpotent test geometry. Its
analytic role in subelliptic estimates and function-space theory on nilpotent
groups is classical~\cite{Folland75}, and refinement and cascade questions on
the Heisenberg group have been studied in harmonic analysis
settings~\cite{LiuLiu06, LiuLiu07}. That body of work concerns refinement
equations and function spaces on the group. Our concern is different. We study a
geometric interpolatory subdivision rule for control polygons that take values
in $\Heis$, with an explicit coordinate-level correction read off the group law.

\subsection{The subdivision operator}

\begin{definition}[Heisenberg four-point operator]\label{def:scheme}
The \emph{Heisenberg four-point subdivision operator} $\SH$ maps
$P^k=\{(x_i^k,y_i^k,z_i^k)\}_{i\in\Z}\subset\Heis$ to
$P^{k+1}=\SH(P^k)$ by,
\begin{enumerate}[label=(\roman*)]
  \item \textbf{Even:} $P_{2i}^{k+1}=P_i^k$.
  \item \textbf{Odd:} $P_{2i+1}^{k+1}=(X_{2i+1},Y_{2i+1},Z_{2i+1})$ where,
    \begin{align}
      X_{2i+1}&=(\Lop x^k)_{2i+1}, \label{eq:Xins}\\
      Y_{2i+1}&=(\Lop y^k)_{2i+1}, \label{eq:Yins}\\
      Z_{2i+1}&=(\Lop z^k)_{2i+1}+R_i^k, \label{eq:Zins}
    \end{align}
    with correction $R_i^k:=\tfrac{1}{2}(x_i^k b_i^k - y_i^k a_i^k)$,
    $a_i^k=X_{2i+1}-x_i^k$, $b_i^k=Y_{2i+1}-y_i^k$.
\end{enumerate}
\end{definition}

\begin{remark}[Origin of the correction, and the (non-)equivariance of the scheme]\label{rem:groupcompat}
The correction $R_i^k$ is the second-order Baker--Campbell--Hausdorff term read
off from the group law. Writing the inserted point as a left translate
$P_{2i+1}^{k+1}=P_i^k\cdot(a_i^k,b_i^k,c_i^k)$ with the horizontal displacement
$(a_i^k,b_i^k)$ fixed by the four-point rule, the law~\eqref{eq:grouplaw}
contributes the cross-term $\tfrac12(x_i^k b_i^k-y_i^k a_i^k)=R_i^k$ to the
central coordinate. Choosing $c_i^k=(\Lop z^k)_{2i+1}-z_i^k$ then
gives~\eqref{eq:Zins}. This is the sense in which the central update is
``derived from the group law.''

We emphasise, however, that $\SH$ as defined is a \emph{coordinate} scheme in
exponential coordinates and is \emph{not} left-equivariant. Under a general left
translation $L_g$, $g=(u,v,w)$, the horizontal data shift by $(u,v)$, and the
correction acquires extra terms $\tfrac12(u\,b_i^k-v\,a_i^k)$. A direct
computation shows $\SH(L_g P)\neq L_g\,\SH(P)$ in general, the discrepancy being
of the same order as the correction itself. The scheme \emph{is} equivariant
under the central one-parameter subgroup $g=(0,0,w)$, which shifts $z$ by $w$
uniformly and commutes with the affine $z$-rule, and it reduces to
$\diag(\Lop,\Lop,\Lop)$ in the abelian limit, where $R_i^k\equiv0$. None of the
analysis below uses left-equivariance. It uses only the explicit coordinate
formulae. The decomposition $\SH=\diag(\Lop,\Lop,\Lop)+N$ is therefore a
statement about coordinate maps, not a proximity relation between an intrinsic
nonlinear scheme and a linear one.
\end{remark}

The present coordinate rule differs from intrinsic log-exp and retraction-based
schemes in a concrete way. In a log-exp construction one maps neighbouring data
into a tangent or Lie-algebra coordinate system, applies a linear mask, and maps
back, after which smoothness-equivalence results compare the nonlinear scheme
with its linear parent~\cite{WallnerNYGrohs07, XieYu10, DuchampXieYu13,
NavYazYu11}. Here the rule is deliberately simpler and more explicit. The
horizontal values are refined linearly, while the central coordinate is
corrected by the Baker--Campbell--Hausdorff term generated by the chosen
basepoint. This simplicity costs left-equivariance, which the scheme does not
possess, but it allows the regularity loss to be tracked exactly, as the
remainder of the paper does.

\section{Explicit formula, Jacobian, and correction estimates}\label{sec:formula}

\subsection{An exact identity for the offset}

The whole analysis rests on the following exact identity, which expresses the
odd-insertion offset $a_i^k:=X_{2i+1}-x_i^k$ in terms of horizontal first
differences $\delta_{x,i}^k=x_{i+1}^k-x_i^k$. The grouping of differences here
is the one that produces the correct scaling limit in
Section~\ref{sec:convergence}.

\begin{lemma}[Difference form of the offset]\label{lem:offset}
For every $\omega\in(0,\tfrac18)$,
\begin{equation}\label{eq:offsetdiff}
  a_i^k = \tfrac12\,\delta_{x,i}^k
        - \omega\,\delta_{x,i+1}^k
        + \omega\,\delta_{x,i-1}^k,
  \qquad
  b_i^k = \tfrac12\,\delta_{y,i}^k
        - \omega\,\delta_{y,i+1}^k
        + \omega\,\delta_{y,i-1}^k.
\end{equation}
In particular the coefficients sum to $\tfrac12$, so under any parametrisation
in which $2^k\delta_{x,j}^k\to X'(t)$ for $j\in\{i-1,i,i+1\}$ one has
$2^k a_i^k\to\tfrac12 X'(t)$, \emph{independently of $\omega$}.
\end{lemma}

\begin{proof}
Write $a_i^k=(\Lop x^k)_{2i+1}-x_i^k
=-\omega x_{i-1}^k+(\omega-\tfrac12)x_i^k+(\tfrac12+\omega)x_{i+1}^k
-\omega x_{i+2}^k$ and substitute
$x_{i-1}^k=x_i^k-\delta_{x,i-1}^k$, $x_{i+1}^k=x_i^k+\delta_{x,i}^k$,
$x_{i+2}^k=x_i^k+\delta_{x,i}^k+\delta_{x,i+1}^k$. The terms in $x_i^k$ cancel
and~\eqref{eq:offsetdiff} follows. The case of $b_i^k$ is identical. The
coefficient sum is $\tfrac12-\omega+\omega=\tfrac12$.
\end{proof}

\subsection{Coordinate formula}

\begin{proposition}[Odd-point formula]\label{prop:formula}
Setting $\Omega(p,q)=p_xq_y-p_yq_x$ for $p,q\in\R^2$,
\begin{equation}\label{eq:fullformula}
\boxed{
\begin{aligned}
X_{2i+1}&={-}\omega x_{i-1}^k+(\tfrac{1}{2}{+}\omega)x_i^k
  +(\tfrac{1}{2}{+}\omega)x_{i+1}^k-\omega x_{i+2}^k,\\[2pt]
Y_{2i+1}&={-}\omega y_{i-1}^k+(\tfrac{1}{2}{+}\omega)y_i^k
  +(\tfrac{1}{2}{+}\omega)y_{i+1}^k-\omega y_{i+2}^k,\\[2pt]
Z_{2i+1}&={-}\omega z_{i-1}^k+(\tfrac{1}{2}{+}\omega)z_i^k
  +(\tfrac{1}{2}{+}\omega)z_{i+1}^k-\omega z_{i+2}^k\\
&\quad -\tfrac{\omega}{2}\Omega(p_i^k,p_{i-1}^k)
  +\tfrac{1}{2}(\tfrac{1}{2}{+}\omega)\Omega(p_i^k,p_{i+1}^k)
  -\tfrac{\omega}{2}\Omega(p_i^k,p_{i+2}^k).
\end{aligned}
}
\end{equation}
For $\omega=1/16$, the $Z$-update is,
\begin{align}\label{eq:Zexplicit}
Z_{2i+1}&=-\tfrac{1}{16}z_{i-1}^k+\tfrac{9}{16}z_i^k+\tfrac{9}{16}z_{i+1}^k
  -\tfrac{1}{16}z_{i+2}^k\notag\\
&\quad+\tfrac{1}{32}(y_i^k x_{i-1}^k-x_i^k y_{i-1}^k)
  +\tfrac{9}{32}(x_i^k y_{i+1}^k-y_i^k x_{i+1}^k)
  +\tfrac{1}{32}(y_i^k x_{i+2}^k-x_i^k y_{i+2}^k).
\end{align}
\end{proposition}

\begin{proof}
By Lemma~\ref{lem:offset},
$a_i^k=\tfrac{1}{2}\delta_{x,i}^k-\omega\delta_{x,i+1}^k+\omega\delta_{x,i-1}^k$
and similarly for $b_i^k$. Substituting into
$R_i^k=\tfrac{1}{2}(x_i^k b_i^k-y_i^k a_i^k)$ and expressing in terms of
$\Omega(p_i^k,p_j^k)$ gives~\eqref{eq:fullformula};~\eqref{eq:Zexplicit}
follows at $\omega=1/16$.
\end{proof}

\subsection{Local Jacobian}

\begin{proposition}[Jacobian]\label{prop:jacobian}
The Jacobian $DF$ of the map $F:(P_{i-1},P_i,P_{i+1},P_{i+2})\mapsto
P_{2i+1}$ has block form $[J_{i-1}\;J_i\;J_{i+1}\;J_{i+2}]$ with
$3\times 3$ blocks,
\begin{align*}
J_{i-1}=J_{i+2}&=
\begin{pmatrix}-\omega&0&0\\0&-\omega&0\\
\tfrac{\omega}{2}y_i^k&-\tfrac{\omega}{2}x_i^k&-\omega\end{pmatrix},\\[4pt]
J_{i+1}&=
\begin{pmatrix}\tfrac{1}{2}{+}\omega&0&0\\0&\tfrac{1}{2}{+}\omega&0\\
-\tfrac{1}{2}(\tfrac{1}{2}{+}\omega)y_i^k
  &\tfrac{1}{2}(\tfrac{1}{2}{+}\omega)x_i^k&\tfrac{1}{2}{+}\omega\end{pmatrix},\\[4pt]
J_i&=
\begin{pmatrix}\tfrac{1}{2}{+}\omega&0&0\\0&\tfrac{1}{2}{+}\omega&0\\
A_i^k&B_i^k&\tfrac{1}{2}{+}\omega\end{pmatrix},
\end{align*}
where $A_i^k=-\tfrac{\omega}{2}y_{i-1}^k+\tfrac{1}{2}(\tfrac{1}{2}+\omega)y_{i+1}^k
-\tfrac{\omega}{2}y_{i+2}^k$ and
$B_i^k=\tfrac{\omega}{2}x_{i-1}^k-\tfrac{1}{2}(\tfrac{1}{2}+\omega)x_{i+1}^k
+\tfrac{\omega}{2}x_{i+2}^k$.
\end{proposition}

\begin{proof}
The $x$- and $y$-rows of each block reproduce the four-point mask coefficients.
The $z$-row entries follow from differentiating $R_i^k=\tfrac{1}{2}(x_i^k b_i^k
-y_i^k a_i^k)$ with respect to each input component.
\end{proof}

\begin{corollary}[Block-triangular structure]\label{cor:triangular}
The sub-blocks $\partial(F_x,F_y)/\partial z_{i+r}$ vanish identically. The
horizontal $2\times 2$ diagonal blocks reproduce exactly the four-point mask.
Perturbations of $z$-data do not affect the horizontal output.
\end{corollary}

\subsection{Correction estimates}

\begin{proposition}[Global bound on $R_i^k$]\label{prop:correctionbounds}
Let $C_0:=\sup_{k,i}(\abs{x_i^k}+\abs{y_i^k})<\infty$. Then,
\begin{equation}\label{eq:globalbound}
  \abs{a_i^k},\;\abs{b_i^k} \leq \tfrac{5}{8}\norm{\Dh P^k}_\infty, \qquad
  \abs{R_i^k} \leq \tfrac{5}{8}C_0\norm{\Dh P^k}_\infty.
\end{equation}
\end{proposition}

\begin{proof}
From the four-point mask,
$\abs{a_i^k}\leq(2\omega+\tfrac{1}{2})\norm{\Dh P^k}_\infty
=\tfrac{5}{8}\norm{\Dh P^k}_\infty$.
Then $\abs{R_i^k}\leq\tfrac{1}{2}C_0(\abs{b_i^k}+\abs{a_i^k})
\leq\tfrac{5}{8}C_0\norm{\Dh P^k}_\infty$.
\end{proof}

\section{Central-channel regularity: convergence without $C^1$}\label{sec:convergence}

This section contains the main analytic result. We first record that the
horizontal channels converge to the classical $C^1$ four-point limit
(Section~\ref{sec:horizontal}). We then derive an \emph{exact} two-scale
recursion for the scaled central first differences
(Section~\ref{sec:recursion}) and show that its forcing term does not decay
(Section~\ref{sec:forcing}). From this, we conclude that the central limit
$Z$ is continuous and lies in the Zygmund class $\Lambda_*$, and that, under an
explicit signed non-cancellation condition on the forcing, it is not $C^1$: the
scaled first differences grow linearly in the refinement level
(Section~\ref{sec:negative}).

Throughout, we assume the initial polygon satisfies the smallness hypotheses,
\begin{equation}\label{eq:hyp}
  \sup_i\norm{P_i^0}\leq C,\quad
  \norm{\Dh P^0}_\infty\leq\varepsilon_0,\quad
  \norm{\Dhh P^0}_\infty\leq\varepsilon_0,
\end{equation}
for a fixed constant $C$ and a sufficiently small $\varepsilon_0>0$ depending
only on $\omega=1/16$, and we set $P^{k+1}=\SH(P^k)$. The smallness
of~$\varepsilon_0$ is a convenient sufficient condition rather than an essential
one. The horizontal channels are exactly linear four-point subdivision and
converge for any bounded initial data, while the central recursion involves the
product terms $x_i^k b_i^k$ and $y_i^k a_i^k$ in $R_i^k$. A bound on the initial
first and second horizontal differences keeps these products and the resulting
forcing uniformly controlled across all levels, which is what the convergence and
Zygmund estimates use. Any hypothesis guaranteeing uniform boundedness of the
horizontal data and their scaled differences would serve equally well. The
negative central-channel conclusion does not depend on the size
of~$\varepsilon_0$, only on the non-vanishing of the limiting forcing $f$.

We use the dyadic ``segment-unit'' parametrisation, in which at level $k$ the node with
index $i$ is assigned the parameter value $t=i\cdot 2^{-k}$, so that an initial
polygon of $m_0$ segments occupies the interval $[0,m_0]$ and one initial
segment has unit length. With this convention, the scaled first differences
$2^k(u_{i+1}^k-u_i^k)$ of a $\Lop$-limit converge to the derivative of that
limit. We write $e_i^k:=2^k(z_{i+1}^k-z_i^k)$ for the scaled central first
differences and $F_i^k:=2^k R_i^k$ for the scaled correction.

\subsection{The horizontal channels}\label{sec:horizontal}

\begin{proposition}[Horizontal $C^1$ limit]\label{prop:horizontal}
Under~\eqref{eq:hyp} the sequences $\{x^k\}$ and $\{y^k\}$ evolve exactly by
the classical four-point operator $\Lop$, independently of the central data,
and their dyadically parametrised interpolants converge uniformly on compact
intervals to functions $X,Y\in C^1(\R)$. Moreover
$2^k\delta_{x,i}^k\to X'(t)$ and $2^k\delta_{y,i}^k\to Y'(t)$ uniformly on
compact intervals (segment-unit parametrisation).
\end{proposition}

\begin{proof}
By Definition~\ref{def:scheme} the $x$- and $y$-updates are $(\Lop x^k)$ and
$(\Lop y^k)$, with no dependence on $z$; this is the content of the
block-triangular structure of Corollary~\ref{cor:triangular}. The conclusion
is then exactly Theorem~\ref{thm:dgl} together with the standard fact that the
first-difference scheme $D_1$ (equivalently its scaled form $S=2D_1$) governs the
scaled first differences, which converge to the ($C^0$) derivative;
see~\cite{Dyn92}, Theorem~3.2.
\end{proof}

Thus, $X$ and $Y$ are precisely the classical four-point limits and carry full
$C^1$ regularity. All of the noncommutative behaviour is confined to the
central channel, to which we now turn.

\subsection{An exact two-scale recursion for the central differences}\label{sec:recursion}

The key structural fact is an exact identity relating the central first
differences at consecutive levels. It is obtained directly from the even and
odd insertion rules, with no approximation.

\begin{proposition}[Exact difference recursion]\label{prop:erecursion}
Write $d_i^k:=z_{i+1}^k-z_i^k$. Then for all $i$ and all $k\geq 0$,
\begin{align}
  d_{2i}^{k+1}   &= \phantom{-}R_i^k
    + \tfrac12 d_i^k - \omega\, d_{i+1}^k + \omega\, d_{i-1}^k,
    \label{eq:evengap}\\
  d_{2i+1}^{k+1} &= -R_i^k
    + \tfrac12 d_i^k + \omega\, d_{i+1}^k - \omega\, d_{i-1}^k.
    \label{eq:oddgap}
\end{align}
Equivalently, in terms of the scaled differences $e^k=2^k d^k$ and the scaled
correction $F^k=2^k R^k$, multiplying~\eqref{eq:evengap}--\eqref{eq:oddgap} by
$2^{k+1}$ gives the exact two-scale recursion,
\begin{equation}\label{eq:erecursion}
  e^{k+1} = S\,e^{k} + \hat g^{k+1},
  \qquad
  \hat g_{2i}^{k+1} = 2 F_i^k,
  \quad
  \hat g_{2i+1}^{k+1} = -2 F_i^k,
\end{equation}
where $S=2D_1$ is the \emph{scaled derived operator}~\eqref{eq:scaledderiv},
\emph{not} the contractive second-difference scheme $D_2$. The factor $2$ that
turns $D_1$ into $S$ is exactly the dyadic rescaling $e^k=2^k d^k$, and it is
what makes the central diagonal of $S$ equal to $1$ rather than $\tfrac12$.
\end{proposition}

\begin{proof}
By the even rule $z_{2i}^{k+1}=z_i^k$ and by the odd
rule~\eqref{eq:Zins} $z_{2i+1}^{k+1}=(\Lop z^k)_{2i+1}+R_i^k$. Hence, the
even-to-odd gap is
\[
  d_{2i}^{k+1}=z_{2i+1}^{k+1}-z_{2i}^{k+1}
  =(\Lop z^k)_{2i+1}-z_i^k+R_i^k .
\]
Now $(\Lop z^k)_{2i+1}-z_i^k=-\omega z_{i-1}^k+(\omega-\tfrac12)z_i^k
+(\tfrac12+\omega)z_{i+1}^k-\omega z_{i+2}^k$, which by the substitution
$z_{i\pm1}^k,z_{i+2}^k\mapsto z_i^k\pm(\text{differences})$ used in
Lemma~\ref{lem:offset} equals $\tfrac12 d_i^k-\omega d_{i+1}^k+\omega
d_{i-1}^k$. This gives~\eqref{eq:evengap}. The odd-to-even gap is
\[
  d_{2i+1}^{k+1}=z_{2i+2}^{k+1}-z_{2i+1}^{k+1}
  =z_{i+1}^k-(\Lop z^k)_{2i+1}-R_i^k
  =\tfrac12 d_i^k+\omega d_{i+1}^k-\omega d_{i-1}^k-R_i^k,
\]
which is~\eqref{eq:oddgap}. Now multiply by $2^{k+1}$. Using $e_j^k=2^k d_j^k$
and $2^{k+1}d_j^k=2e_j^k$, the difference-mask part of~\eqref{eq:evengap}
becomes $2^{k+1}(\tfrac12 d_i^k-\omega d_{i+1}^k+\omega d_{i-1}^k)
=e_i^k-2\omega e_{i+1}^k+2\omega e_{i-1}^k=(Se^k)_{2i}$, and similarly the
odd row gives $(Se^k)_{2i+1}$; compare~\eqref{eq:scaledderiv}. The corrections
contribute $2^{k+1}(\pm R_i^k)=\pm 2 F_i^k$. This is~\eqref{eq:erecursion}.
\end{proof}

\begin{remark}[The forcing is undifferenced, and $S$ is not contractive]\label{rem:undifferenced}
Two features of~\eqref{eq:erecursion} drive the negative result. The forcing is
the scaled correction $F_i^k$ itself, carrying a fixed sign on even nodes and the
opposite sign on odd nodes, and it is \emph{not} a spatial difference of
corrections. Were the forcing a difference $F_i^k-F_{i-1}^k$, it would be
$O(2^{-k})$ and summable, and contractivity would then force a $C^1$ limit. The
genuine forcing $\pm 2F_i^k$ does not vanish (Lemma~\ref{lem:forcing}). Compounding
this, the operator acting on $e^k$ is the scaled operator $S$ rather than the
contractive $D_2$, and $S$ is only marginally stable. A nondecaying forcing
driven through a marginally stable operator produces linear growth, neither
bounded nor exponential, which is the content of Theorem~\ref{thm:nonC1}.
\end{remark}

\subsection{The forcing does not decay}\label{sec:forcing}

\begin{lemma}[Nonvanishing limit of the scaled correction]\label{lem:forcing}
Under~\eqref{eq:hyp}, $F_i^k=2^k R_i^k$ converges uniformly under dyadic
parametrisation $t=i\cdot 2^{-k}$ to,
\begin{equation}\label{eq:flimit}
  f(t)=\tfrac14\bigl(X(t)Y'(t)-Y(t)X'(t)\bigr),
\end{equation}
and this limit is independent of $\omega\in(0,\tfrac18)$. In particular
$\sup_i\abs{F_i^k}\to\norm{f}_\infty$, which is nonzero unless the horizontal
limit lies on a single line through the origin (constant polar angle modulo
$\pi$; see Remark~\ref{rem:generic}).
\end{lemma}

\begin{proof}
Write $F_i^k=\tfrac12(x_i^k\cdot 2^k b_i^k-y_i^k\cdot 2^k a_i^k)$. By the
difference identity of Lemma~\ref{lem:offset},
$2^k a_i^k=\tfrac12\,(2^k\delta_{x,i}^k)-\omega\,(2^k\delta_{x,i+1}^k)
+\omega\,(2^k\delta_{x,i-1}^k)$. By Proposition~\ref{prop:horizontal} each
scaled difference $2^k\delta_{x,j}^k\to X'(t)$ uniformly for
$j\in\{i-1,i,i+1\}$ at $t=i\cdot 2^{-k}$, so,
\[
  2^k a_i^k\to\bigl(\tfrac12-\omega+\omega\bigr)X'(t)=\tfrac12 X'(t),
\]
and likewise $2^k b_i^k\to\tfrac12 Y'(t)$. Since $x_i^k\to X(t)$ and
$y_i^k\to Y(t)$,
\[
  F_i^k\to\tfrac12\Bigl(X(t)\cdot\tfrac12 Y'(t)-Y(t)\cdot\tfrac12 X'(t)\Bigr)
  =\tfrac14\bigl(XY'-YX'\bigr)(t).
\]
The coefficient $\tfrac12$ in $2^k a_i^k\to\tfrac12 X'$ is $\omega$-independent
because the $\pm\omega$ terms cancel, and hence so is $f$.
\end{proof}

\begin{remark}[Correction to a previously claimed value]\label{rem:correction-to-claim}
An earlier analysis of this scheme reported the central derivative correction
as $\tfrac{3}{16}(XY'-YX')$, obtained from an erroneous value
$2^k a_i^k\to\tfrac34 X'$. The correct cancellation in
Lemma~\ref{lem:offset} gives $\tfrac12 X'$ and hence the coefficient
$\tfrac14$ above. More importantly, $f$ is not the derivative of a $C^1$
correction at all. It is the nondecaying amplitude of the oscillatory forcing
in~\eqref{eq:erecursion}, as we now explain.
\end{remark}

\subsection{The detail (Haar) coefficients and the two-scale transport}\label{sec:detail}

The mechanism behind the loss of $C^1$ is cleanest in terms of the local
\emph{detail} coefficients of the scaled differences, which measure the
within-pair oscillation injected at each refinement. This is the precise
two-scale object replacing any informal frequency-mode language. The
coefficients introduced below play the role of local Haar-type details. We use
this language only as an analytic device, and no full wavelet transform is
required. Similar detail-based viewpoints appear in multiscale representations
and in interpolatory wavelets for manifold-valued
data~\cite{GrohsWallner09, Meyer92, RahmanEtal05}.

\begin{lemma}[Mean--detail split of the scaled recursion]\label{lem:detail}
Define the pair-mean and pair-detail of $e^{k+1}$ by,
\[
  m_i^{k+1}:=\tfrac12\bigl(e_{2i}^{k+1}+e_{2i+1}^{k+1}\bigr),\qquad
  \alpha_i^{k+1}:=\tfrac12\bigl(e_{2i}^{k+1}-e_{2i+1}^{k+1}\bigr).
\]
Then, from the exact recursion~\eqref{eq:erecursion},
\begin{equation}\label{eq:meandetail}
  m_i^{k+1}=e_i^{k},
  \qquad
  \alpha_i^{k+1}=2F_i^{k}-\tfrac18\bigl(e_{i+1}^{k}-e_{i-1}^{k}\bigr).
\end{equation}
The pair-mean is exactly the parent value, so the coarse profile is transported
unchanged, while the detail is the forcing amplitude $2F_i^k$ together with a
term of second-difference scale.
\end{lemma}

\begin{proof}
Substitute the even and odd entries of~\eqref{eq:erecursion},
$e_{2i}^{k+1}=2F_i^k+(Se^k)_{2i}$ and $e_{2i+1}^{k+1}=-2F_i^k+(Se^k)_{2i+1}$,
into the definitions. Using~\eqref{eq:scaledderiv}, the means combine as
$\tfrac12\bigl((Se^k)_{2i}+(Se^k)_{2i+1}\bigr)=e_i^k$, with the forcing
cancelling, and the details as
$\tfrac12\bigl((Se^k)_{2i}-(Se^k)_{2i+1}\bigr)
=\tfrac12\bigl(4\omega e_{i-1}^k-4\omega e_{i+1}^k\bigr)
=\tfrac18(e_{i-1}^k-e_{i+1}^k)$ at $\omega=1/16$, while the forcing adds
$\tfrac12(2F_i^k-(-2F_i^k))=2F_i^k$. This gives~\eqref{eq:meandetail}.
\end{proof}

Equation~\eqref{eq:meandetail} is the exact transfer law the analysis needs. The
coarse profile is carried forward verbatim, and at every level, a fresh detail of
amplitude $2F_i^k\to 2f$ is added, perturbed only by a second-difference-scale
term that the Zygmund bound below will show is bounded
(Corollary~\ref{cor:detailbound}). Reconstructing $e^k$ from its details along a
dyadic path is what produces the linear growth, under a transparent
non-cancellation condition made precise in Theorem~\ref{thm:nonC1}.

\subsection{The central limit: continuity, Zygmund regularity, and loss of $C^1$}\label{sec:negative}

We first establish convergence and continuity, then the Zygmund upper bound,
then the loss of $C^1$ under an explicit condition.

\begin{proposition}[Uniform convergence of $Z$]\label{prop:zconv}
Under~\eqref{eq:hyp} the central data $\{z^k\}$ are uniformly bounded, and their
dyadically parametrised interpolants converge uniformly on compact intervals to
a continuous limit $Z$. Moreover $\norm{\Delta z^k}_\infty\to0$ at the
geometric rate $\norm{\Delta z^k}_\infty\leq C_z\,(\tfrac58)^k$.
\end{proposition}

\begin{proof}
From the even and odd rules, every level-$(k+1)$ first difference is one of
$d_{2i}^{k+1}$ or $d_{2i+1}^{k+1}$ in~\eqref{eq:evengap}--\eqref{eq:oddgap}.
Taking absolute values, each is bounded by
$(\tfrac12+2\omega)\norm{\Delta z^k}_\infty+\abs{R_i^k}$. With $\omega=1/16$
the mask part contributes the factor $\tfrac12+2\omega=\tfrac58<1$, and by
Proposition~\ref{prop:correctionbounds} together with the geometric decay of
$\norm{\Dh P^k}_\infty$ (a consequence of horizontal $C^1$ convergence,
Proposition~\ref{prop:horizontal}) we have $\abs{R_i^k}\leq A\cdot 2^{-k}$ for a
constant $A$. Hence,
\[
  \norm{\Delta z^{k+1}}_\infty
  \leq \tfrac58\,\norm{\Delta z^k}_\infty + A\cdot 2^{-k}.
\]
Since $\tfrac58<1$ and the inhomogeneous term is summable, a standard comparison
gives $\norm{\Delta z^k}_\infty\leq C_z(\tfrac58)^k$.

For the uniform Cauchy property, we compare the level-$k$ and level-$(k+1)$
interpolants. At even nodes $z_{2i}^{k+1}=z_i^k$, so the interpolants agree
there. At an odd node $2i+1$ the level-$(k+1)$ value is
$z_{2i+1}^{k+1}=(\Lop z^k)_{2i+1}+R_i^k$, whereas the level-$k$ interpolant
takes the midpoint value $\tfrac12(z_i^k+z_{i+1}^k)$. Their difference is,
\[
  z_{2i+1}^{k+1}-\tfrac12(z_i^k+z_{i+1}^k)
  = \bigl[(\Lop z^k)_{2i+1}-\tfrac12(z_i^k+z_{i+1}^k)\bigr] + R_i^k,
\]
and the bracket is, by direct computation,
\[
  (\Lop z^k)_{2i+1}-\tfrac12(z_i^k+z_{i+1}^k)
  = \omega\bigl((z_i^k-z_{i-1}^k)-(z_{i+2}^k-z_{i+1}^k)\bigr),
\]
a fixed combination of two first differences, hence bounded by
$2\omega\norm{\Delta z^k}_\infty$. Therefore,
\[
  \norm{Z^{k+1}-Z^k}_\infty
  \;\leq\; 2\omega\,\norm{\Delta z^k}_\infty + \norm{R^k}_\infty
  \;\leq\; \tfrac18\,C_z(\tfrac58)^k + A\cdot 2^{-k},
\]
which is summable in $k$. Hence, the interpolants form a uniform Cauchy sequence
and the uniform limit $Z$ is continuous. (The naive bound
$\tfrac12\norm{\Delta z^k}_\infty$, which would hold if odd nodes were exact
midpoints, is \emph{not} available here because of the correction $R_i^k$; the
estimate above is the correct replacement.)
\end{proof}

The decay $\norm{\Delta z^k}_\infty=O\bigl((\tfrac58)^k\bigr)$ guarantees a
continuous limit, but it is weaker than the $O(2^{-k})$ decay that would be
needed for a Lipschitz derivative. The obstruction to $C^1$ lives one
derivative higher, in the scaled differences $e^k$, which we now analyse.

\begin{theorem}[Horizontal regularity and central continuity]\label{thm:cont}
Under~\eqref{eq:hyp} the limit curve $\Gamma=(X,Y,Z)$ exists, with
$X,Y\in C^1(\R)$ generated by $\Lop$ from $\{x_i^0\}$, $\{y_i^0\}$, and
$Z\in C^0(\R)$ the uniform limit of Proposition~\ref{prop:zconv}.
\end{theorem}

\begin{proof}
The horizontal statement is Proposition~\ref{prop:horizontal}. The central
statement is Proposition~\ref{prop:zconv}.
\end{proof}

The passage from a uniform dyadic second-difference bound to membership in the
continuous Zygmund class is standard, and we isolate it as a lemma so that the
proof of Theorem~\ref{thm:zygmund} can quote it cleanly.

\begin{lemma}[Dyadic characterisation of the Zygmund class]\label{lem:dyadiczyg}
Let $Z^k$ be the nested piecewise-linear interpolants with nodal values $z_i^k$
at the dyadic nodes $i2^{-k}$, and suppose $Z^k\to Z$ uniformly on compact
intervals. If there is $C<\infty$ with,
\begin{equation}\label{eq:dyadichyp}
  \sup_i\abs{z_{i+2}^k-2z_{i+1}^k+z_i^k}\leq C\,2^{-k}\qquad(k\ge0),
\end{equation}
then $Z\in\Lambda_*$, with
$\sup_x\abs{Z(x+h)-2Z(x)+Z(x-h)}\le 8C\,\abs h$ for all $h$.
\end{lemma}

\begin{proof}
For an interpolatory refinement, the even nodes persist, so the nodal values are
exact limit values, $z_i^k=Z(i2^{-k})$. Fix $x$ and $h$ with
$2^{-k-1}\le\abs h<2^{-k}$, and set $i=\lfloor x2^k\rfloor$, so that $x$ and
$x\pm h$ lie within two dyadic cells of the node $i2^{-k}$ at level $k$. On each
cell, the interpolation error of $Z^k$ is controlled by the second differences at
the finer levels $m\ge k$. Summing the telescoping contributions of those levels
and using~\eqref{eq:dyadichyp} at each, $\abs{Z-Z^k}\le \tfrac14\sum_{m\ge
k}C\,2^{-m}\le \tfrac12 C\,2^{-k}$ on the relevant cells. Writing
$Z(x+h)-2Z(x)+Z(x-h)$ as the corresponding combination of nodal values plus
three interpolation errors of this size, it is a bounded linear combination of at
most six level-$k$ second differences together with errors of the same order,
each bounded by $C2^{-k}$. Hence
$\abs{Z(x+h)-2Z(x)+Z(x-h)}\le 8C\,2^{-k}\le 8C\cdot 2\abs h$, which after
absorbing the factor gives the stated bound with constant $8C$. The associated
modulus of continuity is $O(h\log(1/h))$; see~\cite{Zygmund}, Ch.~II, Thm.~3.4,
and the dyadic-interpolant form in~\cite{DynLevin02}, \S3.
\end{proof}

\begin{theorem}[Zygmund upper bound]\label{thm:zygmund}
Under~\eqref{eq:hyp}, the scaled second differences of the central data are
uniformly bounded, and there is $K_0<\infty$ with,
\begin{equation}\label{eq:seconddiffbound}
  \sup_i\abs{z_{i+2}^k-2z_{i+1}^k+z_i^k}\leq K_0\,2^{-k}\qquad(k\ge0).
\end{equation}
Consequently $Z$ belongs to the Zygmund class $\Lambda_*$, and there is $K$ with,
\begin{equation}\label{eq:zygmund}
  \sup_{x}\,\abs{Z(x+h)-2Z(x)+Z(x-h)}\leq K\,\abs{h}\qquad(h\in\R),
\end{equation}
equivalently $Z$ has modulus of continuity $\omega_Z(h)=O(h\log(1/h))$. Moreover,
the scaled first differences satisfy the uniform upper bound,
\begin{equation}\label{eq:Okupper}
  \sup_i\,2^k\abs{z_{i+1}^k-z_i^k}=O(k)\qquad(k\to\infty).
\end{equation}
\end{theorem}

\begin{proof}
Apply $\Delta$ to the scaled recursion~\eqref{eq:erecursion}. Writing
$h^k:=\Delta e^k$ and using the exact intertwining $\Delta S=D_2\Delta$
(equation~\eqref{eq:seconddiff}, the defining property of $D_2$), we obtain,
\begin{equation}\label{eq:hrecursion}
  h^{k+1} = D_2\,h^{k} + \Delta\hat g^{k+1}.
\end{equation}
The forcing here is a difference of $\hat g$, but it is \emph{not} small. From
$\hat g_{2i}^{k+1}=2F_i^k$, $\hat g_{2i+1}^{k+1}=-2F_i^k$ the entries of
$\Delta\hat g^{k+1}$ are $-4F_i^k$ (even positions) and $2(F_{i+1}^k+F_i^k)$
(odd positions). Set $M_*:=\sup_{k,i}\abs{F_i^k}$, which is finite because
$F_i^k\to f$ uniformly (Lemma~\ref{lem:forcing}) and the $F_i^k$ are uniformly
bounded. The forcing is then uniformly bounded,
$\norm{\Delta\hat g^{k+1}}_\infty\leq 4\sup_i\abs{F_i^k}\leq 4M_*$, though not
summable. This is harmless because $D_2$ is contractive (Theorem~\ref{thm:dgl},
$\norm{D_2^k}_\infty\leq C_D\mu^k$). Unrolling~\eqref{eq:hrecursion},
\[
  \norm{h^k}_\infty
  \leq C_D\mu^k\norm{h^0}_\infty
   + \sum_{j=1}^{k} C_D\mu^{k-j}\norm{\Delta\hat g^{j}}_\infty
  \leq C_D\norm{h^0}_\infty + \frac{4 C_D M_*}{1-\mu}=:K_0<\infty.
\]
The constant $K_0$ is explicit and computable from the initial data. With the
admissible $C_D,\mu$ of Theorem~\ref{thm:dgl} it is bounded by
$C_D\norm{h^0}_\infty+4C_D M_*/(1-\mu)$.
Since $h_i^k=e_{i+1}^k-e_i^k=2^k(z_{i+2}^k-2z_{i+1}^k+z_i^k)$, the bound
$\norm{h^k}_\infty\leq K_0$ is exactly~\eqref{eq:seconddiffbound}.

For~\eqref{eq:zygmund} we apply Lemma~\ref{lem:dyadiczyg} with $C=K_0$. Its
hypothesis~\eqref{eq:dyadichyp} is exactly~\eqref{eq:seconddiffbound}, and the
interpolants $Z^k$ converge uniformly by Proposition~\ref{prop:zconv}, so
$Z\in\Lambda_*$ with $\sup_x\abs{Z(x+h)-2Z(x)+Z(x-h)}\le 8K_0\abs h$, which
is~\eqref{eq:zygmund} with $K=8K_0$, and $\omega_Z(h)=O(h\log(1/h))$. The
associated function-space viewpoint may be found in~\cite{CMS85, Meyer92}.

Finally~\eqref{eq:Okupper}. The scheme is interpolatory, so even nodes are
preserved at every subsequent refinement, and hence the level-$k$ nodal value is
the exact limit value at the dyadic node,
\begin{equation}\label{eq:nodalupper}
  z_i^k=Z\bigl(i\,2^{-k}\bigr)\qquad(i\in\Z,\ k\ge0).
\end{equation}
Therefore $z_{i+1}^k-z_i^k=Z((i+1)2^{-k})-Z(i2^{-k})$ exactly, and the standard
Zygmund increment estimate, namely that a Zygmund function has modulus of
continuity $O(h\log(1/h))$ (\cite{Zygmund}, Ch.~II), gives with $h=2^{-k}$ the
bound $\abs{z_{i+1}^k-z_i^k}\le C\,2^{-k}\log(2^{k})=C\,k\,2^{-k}$, that is
$2^k\abs{z_{i+1}^k-z_i^k}=O(k)$. No comparison with the convergence rate is
needed, because the discrete difference \emph{is} the limiting chord increment by
interpolation.
\end{proof}

The detail coefficients of Lemma~\ref{lem:detail} are now controlled by the
Zygmund bound, which gives the uniform estimate used in the lower-bound argument.

\begin{corollary}[Bounded detail remainder]\label{cor:detailbound}
With $K_0$ the Zygmund constant of~\eqref{eq:seconddiffbound}, the detail
coefficients of Lemma~\ref{lem:detail} satisfy,
\begin{equation}\label{eq:detailbound}
  \bigl\lvert\,\alpha_i^{k+1}-2F_i^{k}\,\bigr\rvert\le \tfrac{K_0}{4}
  \qquad\text{uniformly in }i,k.
\end{equation}
\end{corollary}

\begin{proof}
By Lemma~\ref{lem:detail}, $\alpha_i^{k+1}-2F_i^k=-\tfrac18(e_{i+1}^k-e_{i-1}^k)
=-\tfrac18(h_i^k+h_{i-1}^k)$, and~\eqref{eq:seconddiffbound} gives
$\abs{h_i^k}\le K_0$, so the modulus is at most $\tfrac18\cdot 2K_0=K_0/4$.
\end{proof}

We now turn to the matching lower bound, the heart of the negative result. It is
where the marginal stability of $S$ and the non-decay of the forcing combine. We
isolate the precise condition under which the linear growth is provable.

\begin{definition}[Local non-cancellation]\label{def:noncancel}
Let $t_0$ have dyadic addresses $i_k=\lfloor t_0 2^k\rfloor$, so that
$i_k\in\{2i_{k-1},2i_{k-1}+1\}$, and let $\epsilon_j\in\{\pm1\}$ be the
\emph{reconstruction signs} determined by the parities of the $i_k$ through the
mean--detail split, so that $\epsilon_j=+1$ if $i_j$ is even and $-1$ if odd, as
detailed in the proof of Theorem~\ref{thm:nonC1}. Say the data satisfy the
\emph{local non-cancellation condition} at $t_0$ if $f(t_0)\neq0$ and,
\begin{equation}\label{eq:noncancel}
  \liminf_{k\to\infty}\;\frac1k\,\Bigl\lvert\,\sum_{j=1}^{k}\epsilon_j\,\alpha_{i_{j-1}}^{\,j}\,\Bigr\rvert
  \;>\;0,
\end{equation}
where $\alpha^{\,j}$ are the detail coefficients of Lemma~\ref{lem:detail}. The
condition is imposed on the \emph{signed} details $\epsilon_j\alpha^j$, the exact
quantities summed in the reconstruction.

A simple but rather restrictive sufficient condition is the following local
remainder bound. By Lemma~\ref{lem:detail}, $\alpha_{i_{j-1}}^{j}
=2F_{i_{j-1}}^{j-1}-r_{i_{j-1}}^{j-1}$ along the path, where
$r_i^{m}:=\tfrac18(e_{i+1}^{m}-e_{i-1}^{m})$ and $F_i^{m}\to f$. Suppose that
along the path to $t_0$ the pathwise remainder is eventually dominated, in the
sense that,
\begin{equation}\label{eq:localdom}
  \limsup_{j\to\infty}\,\abs{r_{i_{j-1}}^{\,j-1}}\;<\;2\abs{f(t_0)}
  \quad\text{and}\quad
  \epsilon_j\,\operatorname{sgn} f(t_0)\ \text{is eventually constant.}
\end{equation}
Then every signed detail eventually has one sign, with an absolute value of at least
$2\abs{f(t_0)}-\limsup_j\abs{r_{i_{j-1}}^{\,j-1}}>0$, so~\eqref{eq:noncancel}
holds. The remainder $r_i^{m}$ is a local second-difference-scale quantity,
typically far smaller than the global Zygmund constant, and in the numerical
examples of Section~\ref{sec:numerics} it is observed to
satisfy~\eqref{eq:localdom} comfortably. The second part of~\eqref{eq:localdom},
the eventual constancy of $\epsilon_j\operatorname{sgn}f(t_0)$, is a genuine
restriction on the binary expansion of $t_0$ and should not be regarded as
typical. The theorem-level hypothesis is~\eqref{eq:noncancel} itself, which is
weaker and is what the numerics actually exhibit.
\end{definition}

\begin{theorem}[Loss of $C^1$ under non-cancellation]\label{thm:nonC1}
Suppose the data satisfy the local non-cancellation
condition~\eqref{eq:noncancel} at some $t_0$. Then the scaled central first
differences grow linearly,
\begin{equation}\label{eq:lineargrowth}
  \sup_i\,2^k\abs{z_{i+1}^k-z_i^k}=\Theta(k)\qquad(k\to\infty),
\end{equation}
and consequently $Z\notin C^1$.
\end{theorem}

\begin{proof}
The upper bound $O(k)$ is~\eqref{eq:Okupper}. For the lower bound, we reconstruct
$e^k$ at the node tracking $t_0$ from the mean--detail
recursion~\eqref{eq:meandetail}. Let $i_k=\lfloor t_0 2^k\rfloor$, so
$i_k\in\{2i_{k-1},2i_{k-1}+1\}$. By the mean--detail split,
$e_{2i}^{k}=m_i^{k}+\alpha_i^{k}=e_i^{k-1}+\alpha_i^{k}$ and
$e_{2i+1}^{k}=m_i^{k}-\alpha_i^{k}=e_i^{k-1}-\alpha_i^{k}$; in either case,
\begin{equation}\label{eq:onestep}
  e_{i_k}^{k}=e_{i_{k-1}}^{k-1}+\epsilon_k\,\alpha_{i_{k-1}}^{k},
  \qquad \epsilon_k=\begin{cases}+1,& i_k\ \text{even},\\ -1,& i_k\ \text{odd},\end{cases}
\end{equation}
the signs $\epsilon_k$ being exactly those of Definition~\ref{def:noncancel}.
Iterating~\eqref{eq:onestep} from level $0$,
\[
  e_{i_k}^{k}=e_{i_0}^{0}+\sum_{j=1}^{k}\epsilon_j\,\alpha_{i_{j-1}}^{j}.
\]
By the non-cancellation hypothesis~\eqref{eq:noncancel}, the signed partial sums
satisfy $\bigl\lvert\sum_{j=1}^k\epsilon_j\alpha_{i_{j-1}}^j\bigr\rvert\ge c\,k-C$
for some $c>0$ and all large $k$. Hence $\abs{e_{i_k}^k}\ge c\,k-C'$, so
$\sup_i\abs{e_i^k}\ge c\,k-C'=\Omega(k)$. With the upper
bound~\eqref{eq:Okupper}, this gives~\eqref{eq:lineargrowth}.

It remains to deduce $Z\notin C^1$. The scheme is interpolatory, so even nodes
are preserved at every subsequent level, and consequently, the nodal value is the
exact limit value,
\begin{equation}\label{eq:nodal}
  z_i^k=Z\bigl(i\,2^{-k}\bigr)\qquad\text{for all }i,k,
\end{equation}
because $z_i^k=z_{2i}^{k+1}=\dots$ stabilises and equals the uniform limit at the
dyadic node $i2^{-k}$. The lower bound above was obtained at the node $i_k$
tracking $t_0$, i.e.\ $\abs{e_{i_k}^k}=2^k\abs{z_{i_k+1}^k-z_{i_k}^k}\ge c\,k-C'$,
and by~\eqref{eq:nodal} this is $2^k\abs{Z((i_k+1)2^{-k})-Z(i_k2^{-k})}$, the
absolute slope of the chord of $Z$ across a dyadic interval shrinking to $t_0$.
If $Z$ were $C^1$ on a neighbourhood of $t_0$, every such chord slope would be
bounded by $\sup\abs{Z'}$ there (mean value theorem); but these slopes grow like
$c\,k\to\infty$, a contradiction. Hence $Z\notin C^1$.
\end{proof}

\begin{remark}[On the lower-bound hypothesis]\label{rem:rigour}
Theorems~\ref{thm:cont}--\ref{thm:zygmund} are unconditional. Uniform
convergence, the membership $Z\in\Lambda_*$, and the $O(k)$ upper bound are all
proved in full, the Zygmund bound through the exact intertwining
$\Delta S=D_2\Delta$ and the contractivity of $D_2$. The linear lower bound
(Theorem~\ref{thm:nonC1}) is conditional on the explicit \emph{signed}
non-cancellation condition~\eqref{eq:noncancel}, imposed on the exact signed sums
$\sum_j\epsilon_j\alpha_{i_{j-1}}^j$ that reconstruct the scaled differences. The
condition is a genuine hypothesis rather than a heuristic, since $\alpha^j$ is an
exactly defined detail coefficient (Lemma~\ref{lem:detail})
and~\eqref{eq:noncancel} is a statement about its signed partial sums along a
dyadic path. A clean sufficient form is the \emph{local} remainder
domination~\eqref{eq:localdom}, in which the small second-difference-scale
remainder $r_i^k=\tfrac18(e_{i+1}^k-e_{i-1}^k)$ along the path stays below
$2\abs{f(t_0)}$. This is a local quantity, much smaller than the global Zygmund
constant, and is the right object to check. The numerical experiments of
Section~\ref{sec:numerics} confirm that the signed partial sums grow linearly,
with a strictly positive and stable slope, for several families of data, and
that the local remainder domination~\eqref{eq:localdom} holds in those cases. We
do not claim that~\eqref{eq:noncancel} is generic in a proven sense, and
establishing openness and density of the signed condition in the
finite-dimensional setting is left open. Localisation matters here, because
globally $\int f$ may vanish, for example for closed curves, so the growth is a
local and path-wise phenomenon, exactly as the numerics show.
\end{remark}

\begin{remark}[Why smoothness is lost]\label{rem:why-lost}
The horizontal channels are $C^1$ because they coincide with the four-point
scheme. The central channel is the four-point scheme \emph{driven} by a forcing
$\pm 2F_i^k$ through the scaled operator $S$. Each refinement injects a fresh
alternating oscillation whose amplitude $2\norm{f}_\infty$ does not shrink with
the level, and $S$ neither amplifies nor damps it, since $S$ is
amplitude-preserving on the alternating sequence. The genuine contraction
appears one derivative higher, in $D_2$, which is the reason the \emph{second}
differences remain bounded, giving the Zygmund regularity, while the \emph{first}
differences accumulate linearly and rule out $C^1$. The group law converts a
benign horizontal mask into a persistently forced and marginally stable central
recursion. Persistent forcing through the neutral component of that recursion is
the borderline that produces Zygmund regularity in place of $C^1$.
\end{remark}

\begin{remark}[The exceptional set, and genericity]\label{rem:generic}
The loss of $C^1$ is driven by $f=\tfrac14(XY'-YX')$, and the construction is
degenerate exactly when $f\equiv0$. Since
$\tfrac{d}{dt}\arg(X+iY)=(XY'-YX')/(X^2+Y^2)$ wherever $(X,Y)\neq(0,0)$, the
condition $XY'-YX'\equiv0$ means the horizontal limit has \emph{constant polar
angle modulo $\pi$}. Its image then lies on a single \emph{line through the
origin}, and not merely on a ray, since the curve may pass through the origin and
reverse radial sign. On this exceptional set the correction vanishes identically
and $Z$ reduces to the $C^1$ four-point limit $Z_0$.

For finite control polygons the exceptional configurations are nongeneric in the
usual finite-dimensional sense. They require all horizontal control points to lie
on a common line through the origin, equivalently that all $2\times2$
determinants $x_i^0 y_j^0-y_i^0 x_j^0$ vanish. This is a finite set of polynomial
equations, so the exceptional set is closed with empty interior, and its
complement is open and dense away from the zero polygon. The phenomenon is
independent of the smallness constants in~\eqref{eq:hyp} and of
$\omega\in(0,\tfrac18)$, because the limiting coefficient $\tfrac14$ in
Lemma~\ref{lem:forcing} is $\omega$-independent.
\end{remark}

\section{Numerical examples}\label{sec:numerics}

All examples use $\omega=1/16$ and $k=7$ refinement steps. The examples follow
the computer-aided geometric design tradition of visualising how a local
refinement rule reshapes control-polygon geometry, in the spirit of numerical
demonstrations for subdivision and for partial differential equation based
geometric design~\cite{DGL87, UgailBW99a, MonterdeUgail04, ShengUgail10}.

\subsection{Example 1: Central-coordinate lifting}

The L-shaped control polygon
$P^0=\{(0,0,0),(1,0,0),(1,1,0),(2,1,0),(2,2,0)\}$
has all central data zero. Figure~\ref{fig:corner} compares the Heisenberg
and Euclidean refined curves.

\begin{figure}[H]
  \centering
  \includegraphics[width=\linewidth]{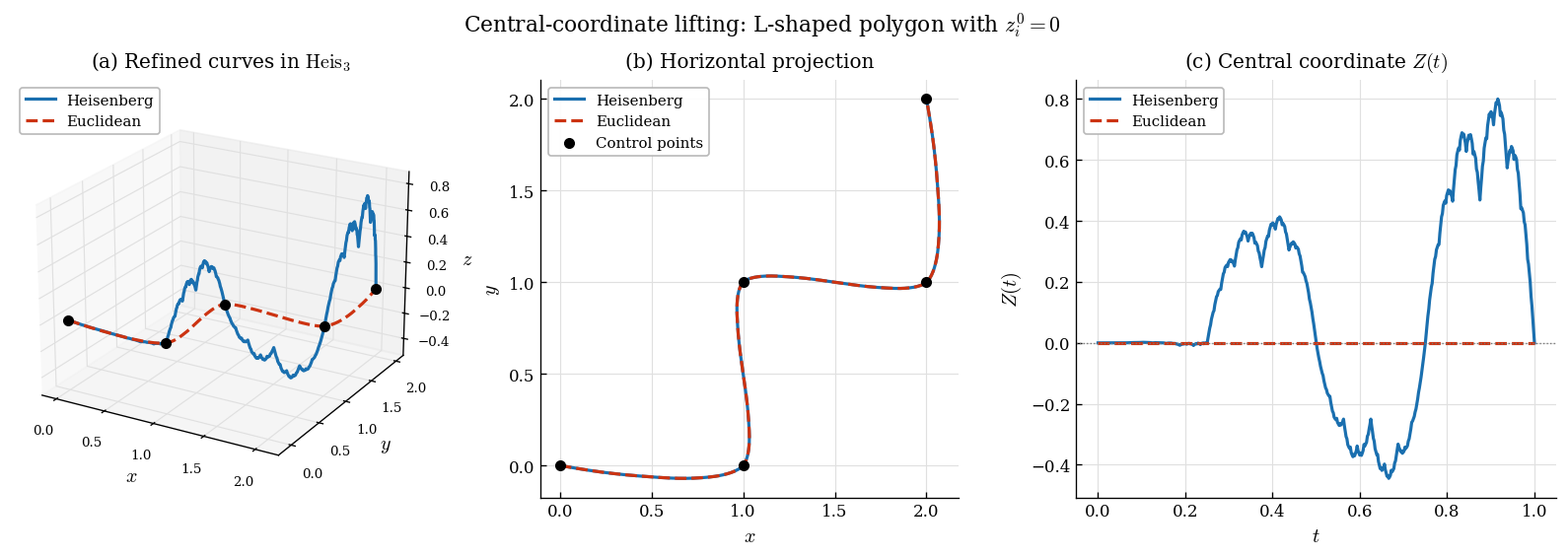}
  \caption{Central-coordinate lifting. \textbf{Left:} 3D refined curves, the
    Euclidean scheme (red dashed) stays in $Z=0$, the Heisenberg scheme (blue)
    lifts away near the corner. \textbf{Middle:} Horizontal projections are nearly
    identical, confirming that the noncommutative correction does not disturb
    the planar geometry (the horizontal limit is the classical $C^1$ four-point
    curve). \textbf{Right:} Central coordinate $Z(t)$, Euclidean gives
    $Z\equiv 0$, Heisenberg generates a nonzero profile. Note the visibly rough,
    non-smooth character of the Heisenberg $Z(t)$. This is not a sampling
    artifact but the signature of the loss of $C^1$ regularity established in
    Theorem~\ref{thm:nonC1}, quantified in Figure~\ref{fig:nonsmooth}.}
  \label{fig:corner}
\end{figure}

\subsection{Example 2: Order sensitivity}

Two control polygons with different traversal order, such that,
\begin{align*}
  \mathcal{A}&:(0,0,0)\to(1,0,0)\to(1.5,0.8,0)\to(2.5,0.8,0)\to(3,1.6,0),\\
  \mathcal{B}&:(0,0,0)\to(0.5,0.8,0)\to(1.5,0,0)\to(2,1.6,0)\to(3,0.8,0).
\end{align*}
Figure~\ref{fig:order} shows that $Z_{\mathcal{A}}(t)$ and $Z_{\mathcal{B}}(t)$
differ in both sign and shape, an effect that no Euclidean componentwise scheme
can produce.

\begin{figure}[H]
  \centering
  \includegraphics[width=0.92\linewidth]{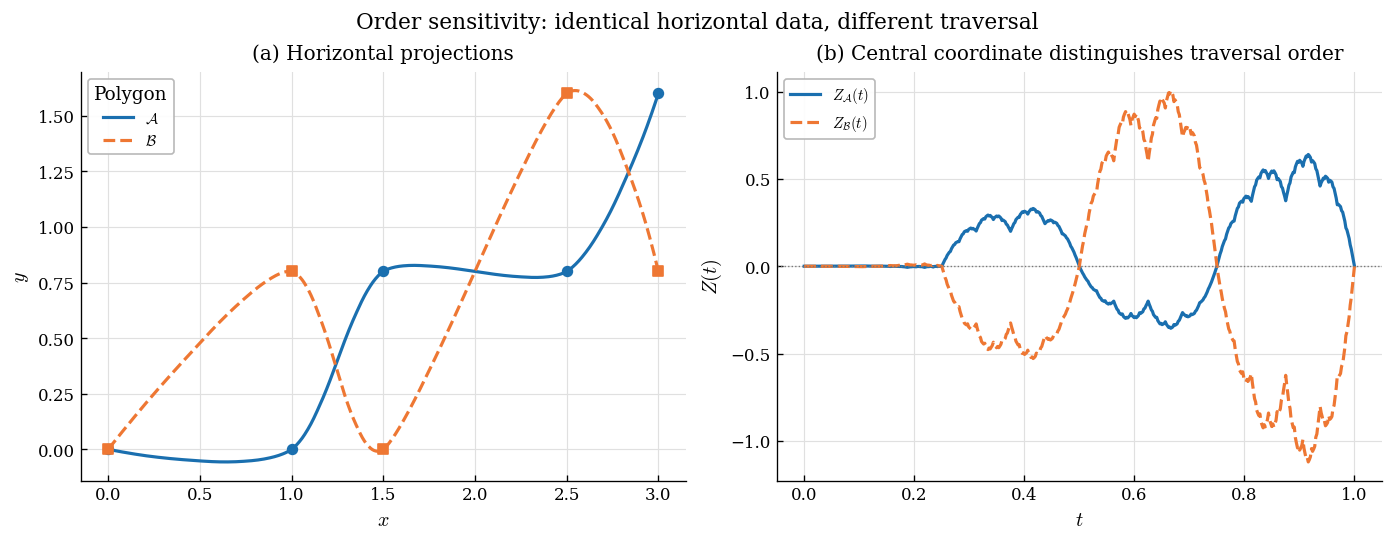}
  \caption{Order sensitivity. \textbf{Left:} Horizontal projections of
    $\mathcal{A}$ (blue) and $\mathcal{B}$ (orange dashed). \textbf{Right:}
    Central coordinates $Z_{\mathcal{A}}(t)$ and $Z_{\mathcal{B}}(t)$ differ
    substantially, the Heisenberg correction distinguishes traversal order.}
  \label{fig:order}
\end{figure}

\begin{table}[H]
  \centering
  \caption{Uniform convergence of $Z^k$ to the level-$12$ reference $Z^*$ for
    the L-shaped data, using piecewise-linear interpolants sampled on the
    reference grid. The ratio $\approx 2$ reflects the $O(2^{-k})$ rate of
    $C^0$ convergence and is consistent with Zygmund (not $C^1$) regularity.}
  \label{tab:convergence}
  \medskip
  \begin{tabular}{@{}ccc@{}}
    \toprule
    Level $k$ & $\sup_t\abs{Z^k-Z^*}$ & Ratio\\
    \midrule
    3&$1.376\times10^{-1}$&---\\
    4&$7.009\times10^{-2}$&1.96\\
    5&$3.530\times10^{-2}$&1.99\\
    6&$1.772\times10^{-2}$&1.99\\
    7&$8.931\times10^{-3}$&1.98\\
    8&$4.271\times10^{-3}$&2.09\\
    \bottomrule
  \end{tabular}
\end{table}

\subsection{Uniform convergence of the central coordinate}

Table~\ref{tab:convergence} records $\sup_t\abs{Z^k(t)-Z^*(t)}$ for the
L-shaped data, where $Z^k$ is the piecewise-linear interpolant of the level-$k$
central values and $Z^*$ is a level-$12$ reference, sampled on the reference
grid. The error halves at each level, confirming the uniform ($C^0$)
convergence of Proposition~\ref{prop:zconv}; Figure~\ref{fig:convergence}
displays the same data on a logarithmic scale. We stress that a ratio
$\approx 2$ here certifies only $C^0$ convergence at rate $O(2^{-k})$; it is
\emph{not} evidence of $C^1$ regularity. Indeed, a $C^0$ Zygmund function has
exactly this $O(2^{-k})$ uniform approximation rate, so Table~\ref{tab:convergence}
and Figure~\ref{fig:convergence} are fully consistent with the non-$C^1$
conclusion of Theorems~\ref{thm:zygmund}--\ref{thm:nonC1}.

\begin{figure}[H]
  \centering
  \includegraphics[width=0.72\linewidth]{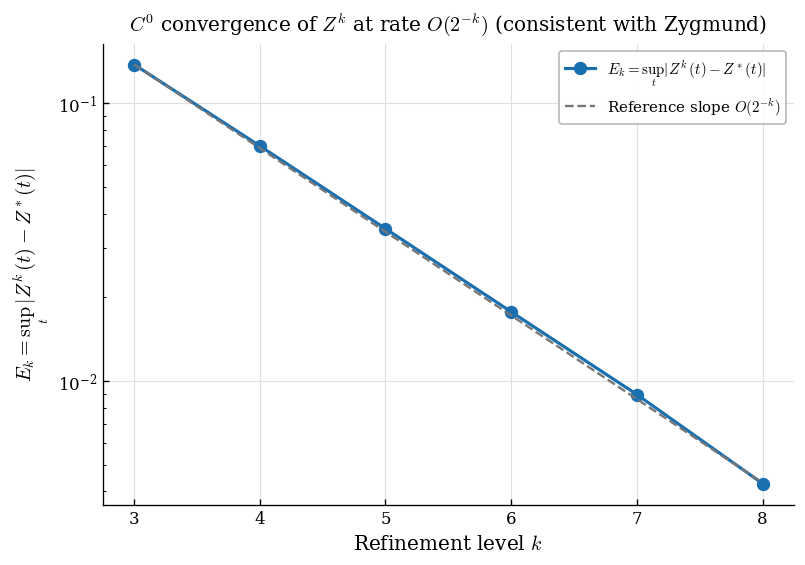}
  \caption{Uniform ($C^0$) convergence of the central coordinate $Z^k$ to the
    level-$12$ reference $Z^*$ for the L-shaped data. The error
    $E_k=\sup_t\abs{Z^k(t)-Z^*(t)}$ decays at the rate $O(2^{-k})$ (the dashed
    reference line has slope $-1$ on the dyadic axis). This rate is exactly that
    of a continuous Zygmund-class limit and is therefore consistent with the
    failure of $C^1$ regularity established in Theorem~\ref{thm:nonC1} and
    Figure~\ref{fig:nonsmooth}; it is not evidence of $C^1$.}
  \label{fig:convergence}
\end{figure}

\begin{figure}[H]
  \centering
  \includegraphics[width=\linewidth]{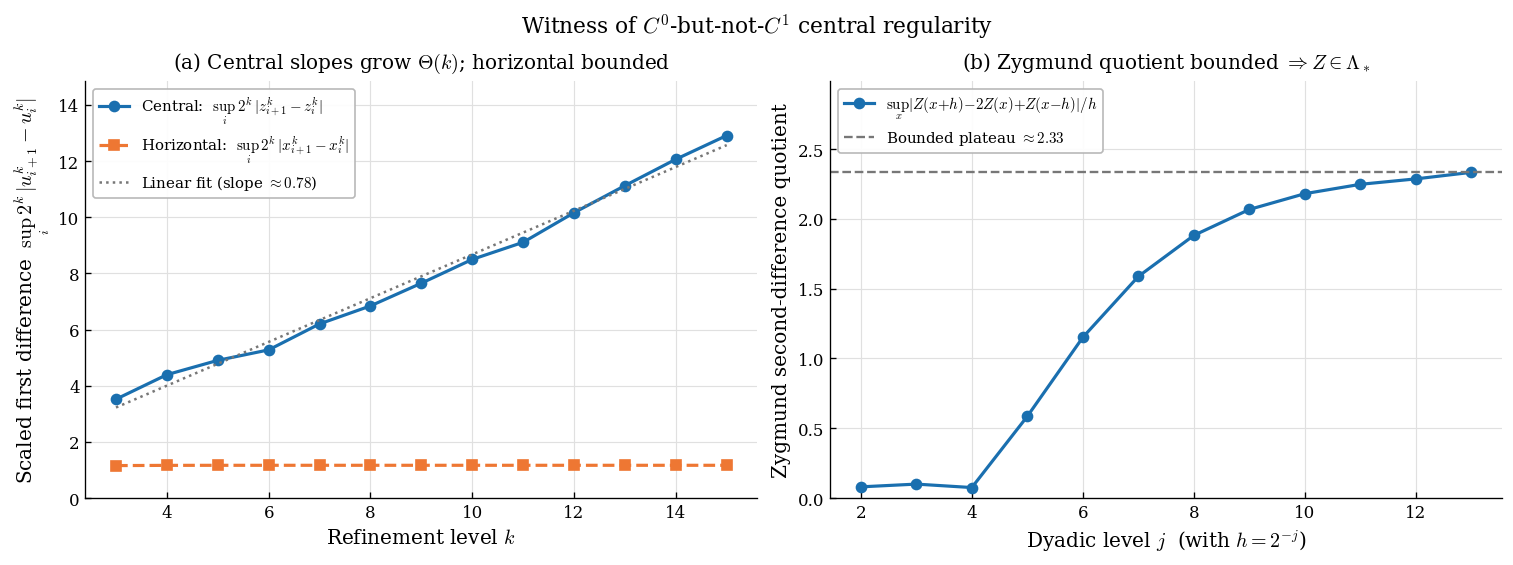}
  \caption{Quantitative witness of $C^0$-but-not-$C^1$ central regularity.
    \textbf{(a)} Scaled central first differences
    $\sup_i 2^k\abs{z_{i+1}^k-z_i^k}$ (blue) grow linearly in the refinement
    level $k$, while the horizontal differences
    $\sup_i 2^k\abs{x_{i+1}^k-x_i^k}$ (orange) stay bounded, and the dotted line is
    a least-squares linear fit over $k=3,\dots,15$ with slope $\approx0.78$ and
    residuals below $3\%$ of the fitted range, supporting $\Theta(k)$ growth,
    which is incompatible with $C^1$. The data are the L-shaped polygon of Example~1,
    open scheme with one-sided index clamping at the two endpoints.
    \textbf{(b)} The Zygmund second-difference quotient
    $\sup_x\abs{Z(x+h)-2Z(x)+Z(x-h)}/h$, $h=2^{-j}$, stays bounded as
    $h\to0$, saturating at a finite plateau $\approx2.3$, numerically illustrating the
    bounded Zygmund quotient proved in Theorem~\ref{thm:zygmund} ($Z\in\Lambda_*$,
    modulus $O(h\log(1/h))$). To remove boundary effects the
    Zygmund test uses the \emph{periodic} scheme on the smooth closed datum
    $X(s)=\cos s+\tfrac15\sin2s$, $Y(s)=\sin s+\tfrac{3}{20}\cos3s$,
    $Z(s)=\tfrac1{10}\cos s$, $s\in[0,2\pi)$, sampled from an initial $16$-gon
    and refined to level $16$; the quotient is evaluated at dyadic scales
    $j=2,\dots,13$.}
  \label{fig:nonsmooth}
\end{figure}

\subsection{Witnessing the loss of $C^1$ regularity}
The decisive numerical signature of Theorem~\ref{thm:nonC1} is the behaviour of
the \emph{scaled} first differences. For a $C^1$ limit, $\sup_i
2^k\abs{u_{i+1}^k-u_i^k}$ converges to $\norm{u'}_\infty$, whereas for the
central channel it instead grows. Figure~\ref{fig:nonsmooth}(a) plots, for the
L-shaped data, the central quantity $\sup_i 2^k\abs{z_{i+1}^k-z_i^k}$ against
the level $k$, alongside the horizontal quantity
$\sup_i 2^k\abs{x_{i+1}^k-x_i^k}$. The horizontal slopes are bounded, consistent
with $X\in C^1$, while the central slopes grow linearly in $k$ with slope
$\approx 0.78$, consistent with $\sup_i 2^k\abs{z_{i+1}^k-z_i^k}=\Theta(k)$ and
hence with $Z\notin C^1$. The strictly positive and stable slope is the
numerical counterpart of the lower-bound mechanism (Lemma~\ref{lem:detail} and
Theorem~\ref{thm:nonC1}). The detail coefficients of the scaled differences
carry the non-decaying forcing of size about $2f$ through the neutral component
of $S$, and their signed partial sums along a dyadic path accumulate linearly,
which is the non-cancellation condition~\eqref{eq:noncancel} in numerical form.
The linear growth is observed in the interior of the parameter interval, away
from the endpoints. The one-sided index clamping used at the two boundary nodes
is standard and does not affect the interior regularity, which is what
Theorem~\ref{thm:nonC1} concerns.

Figure~\ref{fig:nonsmooth}(b) shows the complementary positive fact. The
second-difference quotient $\sup_x\abs{Z(x+h)-2Z(x)+Z(x-h)}/h$ remains bounded
as $h=2^{-j}\to0$, computed on smooth periodic data to remove boundary effects,
saturating at about $2.3$. This numerically illustrates the bounded Zygmund
quotient proved in Theorem~\ref{thm:zygmund}, so that the limit is exactly one
logarithm short of $C^1$, neither smoother nor rougher.

\section{Discussion and extensions}\label{sec:discussion}

The proximity framework~\cite{WallnerDyn05, Wallner06, Grohs10prox,
GrohsWallner09} gives smoothness transfer when a nonlinear scheme is
quadratically close to a linear reference, in the sense that,
\[
\norm{Su-\tilde{S}u}_\infty = O(\norm{\Delta u}_\infty^2).
\]
The Heisenberg correction studied here does not satisfy such a bound. It is only
linearly controlled by the horizontal increments, and this distinction is not a
technical nuisance. It is the analytic reason why proximity theory does not
apply. Indeed, the central limit is not $C^1$ under the non-cancellation
condition of Theorem~\ref{thm:nonC1}, so no $C^1$ regularity-transfer result can
hold for that channel. The failure of quadratic proximity is therefore a
diagnostic of genuine nonsmoothness. The relevant object is instead the exact
two-scale recursion~\eqref{eq:erecursion}, whose forcing is the undifferenced
scaled group-law correction. Tracking this persistent forcing is what reveals
the logarithmic loss of smoothness.

This example clarifies the boundary between positive smoothness-transfer
results and negative regularity phenomena~\cite{XieYu08, XieYu10}. Proximity,
log-exp, and retraction-based schemes are designed so that the nonlinear rule
remains sufficiently close to its linear parent. In the present coordinate
construction, however, the Baker--Campbell--Hausdorff correction is not
quadratically small in the data increments. The result should therefore not be
read as a failure of the proximity framework, but as an example outside its
hypotheses. In this example, the missing quadratic smallness is exactly what
allows the central channel to lose one logarithmic order of smoothness.

The mechanism is explicit. The classical $C^1$ four-point mask is transplanted
unchanged into the horizontal coordinates, and these channels keep their
regularity exactly. The group law then contributes the central correction,
\[
R_i^k=\tfrac12(x_i^k b_i^k-y_i^k a_i^k),
\]
the second-order Baker--Campbell--Hausdorff term in exponential coordinates.
The central channel is therefore a four-point scheme driven by an additional
forcing. After scaling first differences, this forcing has amplitude,
\[
2^k R_i^k \longrightarrow \tfrac14(XY'-YX'),
\]
which is generally nonzero and hence does not decay with the refinement level.
It is fed through the scaled operator $S=2D_1$, whose neutral component does not
damp the corresponding detail. Thus a fresh alternating oscillation of
essentially fixed amplitude is injected at every scale. The second differences
remain bounded because they are governed by the genuinely contractive operator
$D_2$, but the first differences can accumulate linearly. This is precisely the
borderline behaviour that leads to Zygmund regularity rather than $C^1$
regularity.

The lesson is cautionary. Compatibility with a noncommutative group law is not,
by itself, a guarantee of smoothness. A correction may be harmless at one
refinement level and still become regularity-destroying when it is repeated at
all levels with non-decaying amplitude. The limiting coefficient in the forcing
is $\tfrac14$, and it is independent of the mask parameter
$\omega\in(0,\tfrac18)$. This independence follows from the cancellation of the
$\pm\omega$ terms in the limit of $2^k a_i^k$; see Lemma~\ref{lem:offset}. The
loss of $C^1$ is therefore not an artefact of the particular choice
$\omega=1/16$.

The present construction should also be distinguished from geodesic,
log-exp, and retraction-based Lie-group subdivision schemes
~\cite{ParkRavani95, CrouchLeite95, CrouchKunLeite99, NavYazPol13,
WallnerNYGrohs07, XieYu10, NavYazYu11, DuchampXieYu13}. Those schemes are
typically designed so that the nonlinear refinement inherits the regularity of
a linear reference under suitable smoothness-equivalence or proximity
hypotheses. Our scheme is different. It is interpolatory, coordinate-defined,
and not left-equivariant; its horizontal components evolve by the standard
Euclidean four-point symbol, while its central component is corrected directly
by the group law. The negative result should therefore be read as a warning for
coordinate or basepointed group-law constructions, not as a negative statement
about all Lie-group subdivision schemes. The Heisenberg group admits
left-invariant Riemannian metrics and has a natural sub-Riemannian geometry
given by the horizontal distribution spanned by $X$ and $Y$, but the present
scheme does not depend on a particular metric.

A natural next question is how to modify the construction so that the central
channel becomes $C^1$. The analysis points to a clear requirement: the forcing
in the scaled first-difference recursion must be summable across refinement
levels. This could be achieved, for example, by attenuating the BCH correction
with scale, or by replacing the full correction with a genuinely second-order
difference-of-corrections term. A corrected Heisenberg scheme whose nonlinear
part is quadratically close to a linear reference would fall back within the
scope of proximity-type regularity transfer~\cite{Grohs10prox, GrohsWallner09,
DynSharon17a, HuningWallner19}. Designing such a scheme, proving its
regularity, and identifying the resulting central derivative remain natural
problems for future work. Higher-order horizontal masks, such as six-point or
eight-point schemes, would improve the regularity of $X$ and $Y$, but they
would not by themselves remove the central obstruction, since the obstruction is
caused by the non-decay of the forcing rather than by insufficient smoothness of
the horizontal mask.

The results in this paper concern curves. A surface analogue could be built by
applying four-point refinement in two parameter directions and adding group-law
corrections at inserted vertices. The present analysis suggests that the
central channel would again be the delicate part. This question is relevant to
geometric design, where subdivision, partial differential equation surfaces,
B\'ezier patches, and mesh approximation provide complementary ways of
constructing smooth geometry from sparse control data~\cite{BloorWilson89,
UgailBW99a, KubiesaUW04, MonterdeUgail06, GonzalezUgail08, ShengUgail10}.
Finally, the order-sensitive correction $R_i^k$ remains a natural
noncommutative building block for sequential data. Its naive multiscale
iteration, however, should be used with care: it produces Zygmund-class output
rather than a smooth central curve.

\section{Conclusion}\label{sec:conclusion}

In this paper, we have introduced an interpolatory subdivision scheme for control polygons
with values in the three-dimensional Heisenberg group and analysed the
regularity of the resulting limit curve. The construction is deliberately
explicit. The two horizontal coordinates are refined by the classical four-point
scheme, while the central coordinate is modified by a closed-form correction
derived from the Heisenberg group law. Because of the block-triangular structure
of the refinement rule, the horizontal limit is exactly the classical
four-point limit and retains its $C^1$ regularity. Thus, the genuinely
noncommutative effect appears only in the central channel.

The central-channel analysis is governed by an exact two-scale recursion for
the scaled central first differences. The forcing term in this recursion is the
scaled group-law correction itself, not a difference of corrections. Its
amplitude converges to a generally nonzero limit and therefore persists across
refinement levels. Passing to second differences, and using the intertwining
between the scaled first-difference operator and its contractive
second-difference scheme gives the positive part of the result: the central
coordinate converges uniformly to a continuous limit in the Zygmund class, with
modulus of continuity of order $h\log(1/h)$. Equivalently, the scaled first
differences satisfy the unconditional upper bound $O(k)$ at refinement level
$k$. Under an explicit signed non-cancellation condition on the detail
coefficients, a matching linear lower bound holds. In that case, the scaled
first differences grow like $\Theta(k)$, and the central limit is not
continuously differentiable. The numerical examples illustrate this behaviour
for typical nondegenerate data.

The broader conclusion is cautionary. Compatibility with a noncommutative group
law is not, by itself, a guarantee of smoothness. A correction that is benign at
one refinement level can destroy the regularity of the limit when it is injected
with non-decaying amplitude at every scale. This does not contradict proximity or
log-exp smoothness-transfer theory; rather, it identifies a mechanism outside
their hypotheses. In the present coordinate scheme, the correction is only
linearly close to the Euclidean reference rule, whereas proximity-based
$C^1$ transfer requires a quadratic closeness estimate. Natural next steps are
to design a modified Heisenberg scheme whose central correction is
second-order close to a linear reference, to extend the analysis to surface
subdivision, and to determine whether the signed non-cancellation condition
holds generically.

\section*{Acknowledgments}
The authors would like to thank the Centre for Visual Computing and Intelligent
Systems at the University of Bradford for providing the computing resources used
in this work.

\section*{Code and data availability}
The complete numerical code, the figure-generation pipeline, and all per-level
and per-condition results required to reproduce every experiment reported in this
paper are openly available at \url{https://github.com/ugail/Heisenberg-Subdivision}.


\end{document}